\documentclass[12pt]{amsart}%
\usepackage{amsmath,amssymb}

\theoremstyle{plain}
\newtheorem{thm}{Theorem}[section]
\newtheorem{theorem}[thm]{Theorem}

\newtheorem{lemma}[thm]{Lemma}
\newtheorem{corollary}[thm]{Corollary}
\newtheorem{proposition}[thm]{Proposition}
\theoremstyle{definition}
\newtheorem{remark}[thm]{Remark}

\newtheorem{definition}[thm]{Definition}

\newtheorem{example}[thm]{Example}

\newtheorem{conjecture}[thm]{Conjecture}

\numberwithin{equation}{section}

\newcommand{\sC}{{\mathcal C}}

\newcommand{\sK}{{\mathcal K}}

\newcommand{\sO}{{\mathcal O}}

\newcommand{\sR}{{\mathcal R}}

\newcommand{\sU}{{\mathcal U}}

\newcommand{\C}{{\mathbb C}}

\newcommand{\N}{{\mathbb N}}
\newcommand{\BP}{{\mathbb P}}

\newcommand{\Z}{{\mathbb Z}}

\newcommand{\GL}{{\rm \mathbf GL}}
\newcommand{\PGL}{{\rm \mathbf PGL}}

\newcommand{\Sym}{{\rm Sym}}

\title[Euler-symmetric projective varieties]{Euler-symmetric projective varieties}
\author{Baohua Fu and Jun-Muk Hwang}
\thanks{Baohua Fu is supported by National Natural Science Foundation of China (11321101 and 11431013).  Jun-Muk Hwang is
supported by National Researcher Program 2010-0020413 of NRF}
\begin{document}
\maketitle  \setcounter{tocdepth}{1}
\begin{abstract}
Euler-symmetric projective varieties are nondegenerate projective varieties admitting many $\mathbb{C}^{\times}$-actions of Euler type. They are quasi-homogeneous and uniquely determined by their fundamental forms at a general point. We show that Euler-symmetric projective varieties can be classified by symbol systems, a class of algebraic objects  modeled on the systems of fundamental forms at general points of projective varieties.  We study  relations between the algebraic properties of symbol systems and the geometric properties of Euler-symmetric projective varieties. We describe also the relation between  Euler-symmetric projective varieties of dimension $n$ and equivariant compactifications of the vector group $\mathbb{G}_a^n$.
\end{abstract}

\section{Introduction}
In \cite{FH2}, the authors  introduced the notion of quadratically symmetric varieties in order to link the study of special birational transformations to the prolongations of linear Lie algebras. A quadratically symmetric variety is  quasi-homogeneous and it is homogeneous if and only if it is one of the Hermitian symmetric spaces of rank 2. Thus we may say that quadratically symmetric varieties are quasi-homogeneous generalizations of Hermitian symmetric spaces of rank 2.

The goal of this article is to introduce Euler-symmetric projective varieties, which are quasi-homogeneous generalizations of Hermitian symmetric spaces of arbitrary ranks. Euler-symmetric projective varieties are
  nondegenerate projective varieties admitting many $\mathbb{C}^{\times}$-actions of Euler type
  (Definition \ref{d.Euler}).

 We show that any Euler-symmetric projective variety is uniquely determined by their fundamental forms at a general point (Proposition \ref{p.isom}). By Cartan's theorem (Theorem \ref{t.Cartan}), the collection of these fundamental forms satisfies the prolongation property. To make this more systematic, we introduce the notion of a symbol system, formalizing the prolongation property.  For any symbol system ${\bf F}$, we construct an Euler-symmetric projective variety $M({\bf F})$ whose  fundamental forms at general points are isomorphic to ${\bf F}$ (Theorem \ref{t.model}). This reduces the classification of Euler-symmetric projective varieties to that of symbol systems.

The relation between the algebraic properties of a symbol system and the geometric properties of the associated Euler-symmetric projective variety is very intriguing.
A key question is which symbol systems give rise to nonsingular Euler-symmetric projective varieties.  We show that for a nonsingular Euler-symmetric projective variety,  the base loci of the symbol system is nonsingular (Proposition \ref{p.vmrt}).

 The most important geometric property of Euler-symmetric projective varieties is that they are equivariant compactifications of  vector groups. Conversely, we show that an $n$-dimensional prime Fano manifold of Picard number one is Euler-symmetric if and only if it is an equivariant compactification of $\mathbb{G}_a^n$ (Corollary \ref{c.SEC}).

Our results show that the interaction between the algebra of a symbol system ${\bf F}$ and the geometry of the Euler-symmetric projective variety  $M({\bf F})$
  is worth investigating. Among others, this will give new insights into fundamental forms of projective varieties. We mention that we have touched on only a small number of issues in this article: there remain a wide range of questions to be explored in this interaction.

\section{Euler-symmetric projective varieties and systems of fundamental forms}

\begin{definition}\label{d.Euler} Let $Z \subset \BP V$ be a projective variety. For a nonsingular point $x\in Z$, a $\C^{\times}$-action on  $Z$ coming from a multiplicative subgroup of $\GL(V)$ is said to be  of {\em Euler type} at $x$, if $x$ is an isolated fixed point of the induced action on $Z$ and the isotropic action on
the tangent space $T_x Z$ is by scalar multiplication (i.e., the induced action on $\BP T_x Z$ is trivial). We say that $Z \subset \BP V$ is {\em Euler-symmetric}
if for a general point $x \in Z$, there exists a $\C^{\times}$-action on $Z$ of Euler type at $x$. \end{definition}

The example below shows that there are at least as many nonsingular Euler-symmetric projective varieties as nonsingular projective varieties.

\begin{example}\label{e.blowup}
Let $S \subset \BP^{n-1} \subset \BP^n$ be a nonsingular algebraic subset in a hyperplane of $\BP^n$. For each point $x \in \BP^n \setminus \BP^{n-1}$, the scalar multiplication on the affine space $\BP^n \setminus \BP^{n-1}$  regarded as a vector space with the origin at $x$ can be extended to  a $\C^{\times}$-action $$A_x: \C^{\times} \times \BP^n \to \BP^n,$$ which fixes every point of the hyperplane $\BP^{n-1}$. Let $\beta: {\rm Bl}_S(\BP^n) \to \BP^n$ be the blowup of $\BP^n$ along $S$ and let $E$ be the exceptional divisor. For  suitable positive integers $a$ and $b$, the line bundle $L:= \sO(-aE) \otimes \beta^* \sO_{\BP^n}(b)$ is very ample. The action $A_x$ induces an action on the image $$ Z \subset \BP H^0({\rm Bl}_S(\BP^n), L)^*$$ of the projective embedding, which is of Euler type at $x \in Z$. Thus $Z$ is an Euler-symmetric projective variety. \end{example}

We will give more examples of Euler-symmetric projective varieties in the next section.

\begin{proposition}\label{p.quasihomo}
 An Euler-symmetric projective variety $Z
\subset \BP V$ is quasi-homogeneous, i.e., the linear automorphism
group ${\rm Aut}(Z) \subset \PGL(V)$ acts on $Z$ with a dense open
orbit.  \end{proposition}

\begin{proof} Let $G \subset \PGL(V)$ be the identity component of the group of
projective automorphisms of $Z$. We need to show that $G$ has an
open orbit on $Z$. By the general structure of an algebraic group
action, there exists a $G$-stable Zariski open subset $Z_o \subset Z$ such
that for any $x\in Z_o$, the intersection of the orbit $G \cdot x$
and $Z_o$ is a closed subset in $Z_o$. For a general $x \in Z_o$,
let $A_x \subset G$ be a multiplicative subgroup inducing a $\C^{\times}$-action of Euler type at $x$. From the Bialynicki-Birula decomposition
theorem for $\C^{\times}$-action(\cite{BB}), there exists an open neighborhood $U
\subset Z_o$ of $x$ with holomorphic coordinates $z_1, \ldots,
z_n, n := \dim Z,$ on $U$ such that the orbits of $A_x$ are radial lines through
$x=(z_1= \cdots = z_n=0)$ in this coordinates. Thus for any point
 $y \in U \setminus \{x\}$, the closure of the orbit $A_x \cdot y \subset G \cdot y$
 contains $x$. Since the $G$-orbit of $y$ is closed in $Z_o$, we see
 that $x \in G \cdot y$, implying $y \in G \cdot x$. Thus $G\cdot
 x$ contains an open subset in $Z$. \end{proof}

To describe Euler-symmetric projective varieties explicitly, it is convenient to use fundamental forms, the definition of which we recall below (see p. 98 of \cite{IL} or Section 2.1 \cite{LM03}).

\begin{definition}\label{d.FF}
Let $x \in Z \subset \BP V$ be a nonsingular point of a nondegenerate projective variety.
Let $L$ be the line bundle on $Z$ given by the restriction of the hyperplane line bundle on $\BP V$. For each nonnegative integer $k$, let ${\bf m}_{x,Z}^k$ be the $k$-th power of the maximal ideal ${\bf m}_{x, Z}.$
For a section  $s\in H^0(Z, L),$ let $j_x^k(s)$ be the $k$-jet of $s$ at $x$ such that $j_x^0(s) = s_x \in L_x$. We have a descending filtration of the dual space $V^* \subset H^0(Z, L)$ by
$$V^* \cap {\rm Ker}(j_x^{k}) \subset V^* \cap {\rm Ker}(j_x^{k-1}).$$ The induced homomorphism
$$(V^* \cap {\rm Ker}(j_x^{k-1})) / (V^* \cap {\rm Ker}(j_x^{k}))  \to L_x \otimes \Sym^{k} T^*_x Z$$ is injective. For each $k \geq 2$, the subspace $F^{k}_x \subset \Sym^k T^*_x Z$ defined by the image of this homomorphism is called the $k$-th {\em fundamental form} of $Z$ at $x$. For convenience, set
 $F^0_x = \Sym^0 T^*_x Z = \C$ and $F^1_x = \Sym^1 T^*_x Z = T^*_x Z$. The collection of subspaces $${\bf F}_x:= \oplus_{k \geq 0}  F^k_x \subset \oplus_{k \geq 0} \Sym^k T^*_x Z$$ is called the {\em system of fundamental forms} of $Z$ at $x$.
\end{definition}

It is straight-forward to translate this definition of fundamental forms into the language of inhomogeneous coordinates, as follows.

\begin{lemma}\label{l.FF}
Let $x \in Z \subset \BP V$ be a nonsingular point of a nondegenerate projective variety.
 We can choose \begin{itemize} \item[(a)]  positive integers  $1 = m_1 < m_2 < \cdots < m_r$ and  $$n_1 =n = \dim Z,  n_2, \ldots, n_r  \mbox{ satisfying } \dim \BP V  = n_1 + \cdots + n_r;$$ \item[(b)] an inhomogeneous coordinate system $$(z^{(1)}_1, \ldots, z^{(1)}_{n_1}, z^{(2)}_1, \ldots, z^{(2)}_{n_2}, \ldots, z^{(r)}_1, \ldots, z^{(r)}_{n_r})$$ on $\BP V$ such that \begin{itemize} \item[(b1)] $x = (z^{(i)}_j =0, 1 \leq i \leq r, 1 \leq j \leq n_i)$; \item[(b2)]
the embedded tangent space ${\bf T}_x Z$  of $Z$ at $x$ is given by   $$ {\bf T}_x Z = (z^{(i)}_j =0, 2 \leq i \leq r, 1 \leq j \leq n_i);$$\end{itemize}
   \item[(c)]   holomorphic functions   $$h^i_j(z_1, \ldots, z_{n}), \ 2 \leq i \leq r, 1 \leq j \leq n_i, $$ in the variables $z_1 := z^{(1)}_1, \ldots, z_n := z^{(1)}_n$ defined near the origin of ${\bf T}_x Z$ such that \begin{itemize} \item[(c1)] the germ of   $Z$ at $x$ is  defined by the equations $$
z^{(i)}_j = h^i_j(z_1, \ldots, z_n), \ 2 \leq i \leq r, 1 \leq j \leq n_i;$$
 \item[(c2)] for each $i, 2\leq i \leq r$,   the lowest order terms of $h^i_j, 1 \leq j \leq n_i, $ are $n_i$ linearly independent  homogeneous polynomials of degree $m_i$ in the variables $z_1, \ldots, z_n.$ \end{itemize} \end{itemize} Then for each $2 \leq i \leq r$, the collection of homogeneous polynomials of degree $m_i$ arising as the lowest order terms of $h^i_j, 1 \leq j \leq n_i$, is exactly the $m_i$-th fundamental form $F^{m_i}_x$ (and $F^k_x =0$ for $k \not\in \{0,1, m_1, \ldots, m_r\}$). \end{lemma}

     \begin{definition}\label{d.isom} Let  $Z_1, Z_2 \subset \BP V$ be two projective varieties of equal dimension.
     Let $x_1 \in Z_1$ and $x_2 \in Z_2$ be nonsingular points. We say that the systems of fundamental forms
     ${\bf F}_{x_1}$ and ${\bf F}_{x_2}$ are {\em isomorphic} if there exists a linear isomorphism
     $\varphi: T^*_{x_1} Z_1 \to T^*_{x_2} Z_2$ such that the induced isomorphism
     $$\oplus_{k \geq 0} \Sym^k T^*_{x_1} Z_1 \to \oplus_{k \geq 0} \Sym^k T^*_{x_2} Z_2$$ sends
     ${\bf F}_{x_1}$ isomorphically to ${\bf F}_{x_2}$.
     \end{definition}

     \begin{proposition}\label{p.isom}
     Let $Z_1$ and $Z_2$ be two Euler-symmetric projective varieties in $\BP V$ of equal dimension.
     Let $x_1 \in Z_1$ and $x_2 \in Z_2$ be general points. If ${\bf F}_{x_1}$ and ${\bf F}_{x_2}$ are isomorphic in the sense of Definition \ref{d.isom}, then $Z_1$ and $Z_2$ are isomorphic by a projective transformation on $\BP V$. \end{proposition}

     \begin{proof} If there exists  a $\C^{\times}$-action on a nondegenerate variety $Z \subset \BP V$ which is of Euler type at a nonsingular point $x \in Z$,  the induced action on $\oplus_{k \geq 0} \Sym^k T^*_x Z$ preserves the subspaces $V^* \cap {\rm Ker}(j^k_x)$ and ${\bf F}_x$  in Definition \ref{d.FF}.  Thus in Lemma \ref{l.FF}, we can choose the inhomogeneous coordinates $z^{(i)}_j$ to be eigenfunctions of
     the $\C^{\times}$-action such that an element $s \in \C^{\times}$ acts by
     $z^{(i)}_j \mapsto s^{d_{ij}} z^{(i)}_j$ for some integer $d_{ij}$. For the germ of $Z$ near $x$ to be preserved under this
     $\C^{\times}$-action, the holomorphic function $h^i_j$ in Lemma \ref{l.FF} (c) must be a homogeneous polynomial of degree $m_i$ (i.e. all higher order terms vanish). Thus such $Z$ is determined by the isomorphism type of the system of fundamental forms at $x$, up to the action of $\GL(V)$. \end{proof}


\section{Euler-symmetric projective variety determined by a symbol system}

\begin{definition}\label{d.prolong}
Let $W$ be a vector space. For $w \in W$, the contraction homomorphism
$\iota_w: \Sym^{k+1} W^* \to \Sym^k W^*$  sends $\varphi \in \Sym^{k+1} W^*$ to
$\iota_w \varphi \in \Sym^k W^*$  defined by
$$ \iota_w\varphi (w_1, \ldots, w_k) = \varphi (w, w_1, \ldots, w_k)$$ for any
$w_1, \ldots, w_k \in W$. By convention, we define $\iota_w(\Sym^0 W^*) =0$.
For a subspace $F \subset \Sym^k W^*$  of symmetric $k$-linear forms on $W$, define its prolongation ${\bf prolong}(F) \subset \Sym^{k+1} W^*$ as the subspace consisting of
symmetric $(k+1)$-linear forms $\varphi $ on $W$ satisfying
$\iota_w \varphi \in F$ for any $w \in W$, i.e.,
$${\bf prolong}(F) := \bigcap_{w \in W} \iota_w^{-1}(F).$$ \end{definition}

\begin{definition}\label{d.symbol}
Let $W$ be a vector space. Fix a natural number $r$.  A subspace $${\bf F} = \oplus_{k \geq 0} F^k \subset \oplus_{k \geq 0} \Sym^k W^*$$  with
$$F^0 = \C = \Sym^0 W^*, F^1 = W^*,  F^r \neq 0, \mbox{ and } F^{r+i} =0 \mbox{ for all } i \geq 1,$$
is called a {\em symbol system of rank} $r$,  if $F^{k+1} \subset {\bf prolong}(F^{k})$ for each $1 \leq k \leq r$, or equivalently, if $\iota_w {\bf F} \subset {\bf F}$ for any $w \in W$. \end{definition}

Natural examples of symbol systems are provided by the following classical result due to E. Cartan (p.68 of \cite{LM03} or  Exercise 3.5.10 of \cite{IL}).

\begin{theorem}\label{t.Cartan}
Let $Z \subset \BP V$ be a nondegenerate subvariety and let $x \in Z$ be a general point.
Then the system of fundamental forms ${\bf F}_x = \oplus_{k \geq 0} F^k_x$ is a symbol system of rank $r$ for some natural number $r \geq 1$. In particular,   we have $n_i \geq 1$ and
$m_i = i$ for all $1 \leq i \leq r$ in Lemma \ref{l.FF}.
\end{theorem}

\begin{remark}
Lemma \ref{l.FF} shows that there is essentially no restriction on fundamental forms at nonsingular points of  projective varieties: any collection of subspaces $F^k \subset \Sym^k W^*, 2 \leq k \leq r,$ can be realized as a system of fundamental forms at some nonsingular point of some projective variety. On the other hand, Theorem \ref{t.Cartan} says that the system of fundamental forms at a general point of a projective variety cannot be arbitrary: it has to be a symbol system. Is there any other restriction? Theorem \ref{t.model} below shows that there is no other restriction. \end{remark}

\begin{lemma}\label{l.dot}
 Let  ${\bf F}$ be  a symbol system as in Definition \ref{d.symbol}. For an element $w \in W$, the $j$-th composition  $\iota_w^j := \iota_w \circ \cdots \circ \iota_w$ sends $ F^k $ to $ F^{k-j}$ for each $k$, setting $F^{-i} =0$ if $i \geq 1$. Then for each positive integer $j$ and $u, v \in W$,
 $$\iota_{v+u}^j = \sum_{l=0}^j  \binom{j}{l}  \iota_{v}^l \circ \iota_{u}^{j-l}.$$
    \end{lemma}

\begin{proof}
For any $\varphi \in F^k$ and $v, u \in W$, we have
\begin{eqnarray*}
\iota_{v+u}^j (\varphi) & = &  \varphi(\underbrace{v+u, \cdots, v+u}_j,  \cdots )
\\ &=&  \sum_{l=0}^j  \binom{j}{l}  \varphi(\underbrace{v, \cdots, v}_l, \underbrace{u, \cdots, u}_{j-l},  \cdots ),
\end{eqnarray*} which implies the desired equality.
\end{proof}

\begin{definition}\label{d.model}
In Lemma \ref{l.dot}, the restriction of $\iota_w^k$ to $F^k$ determines  an element in $(F^k)^*$, which is just the map $\varphi \mapsto \varphi(w, \cdots, w)$. By abuse of notation, we will just denote it by $\iota_w^k \in (F^k)^*$, if no confusion arises.
 Define a rational map $$\phi_{\bf F} : \BP (\C \oplus W) \dasharrow \BP( \C \oplus W \oplus (F^2)^* \oplus \cdots \oplus (F^r)^*)$$
by $$[t: w] \mapsto [t^r: t^{r-1}w: t^{r-2}\ \iota_w^2:  \cdots: t\ \iota_w^{r-1}: \iota_w^r].$$
Write  $V_{\bf F} := \C \oplus W \oplus (F^2)^* \oplus \cdots \oplus (F^r)^*$. We will denote the proper image of the rational map $\phi_{\bf F}$ by $M({\bf F}) \subset \BP V_{\bf F}$.
\end{definition}

\begin{theorem}\label{t.model}
In Definition \ref{d.model}, let $o=[1:0:\cdots:0] \in M({\bf F})$ be the point $\phi_{\bf F} ([t=1: w=0]).$
\begin{itemize} \item[(i)] The natural action of the vector group $W$ on $\BP (\C \oplus W)$ can be extended to an action of $W$ on
$\BP V_{\bf F}$ preserving $M({\bf F})$ such that the orbit of $o$ is an open subset biregular to $W$. \item[(ii)] The $\C^{\times}$-action on
$W$ with weight 1 induces a $\C^{\times}$-action on $M({\bf F})$ of Euler type at $o$, making $M({\bf F})$  Euler-symmetric. \item[(iii)] The system of fundamental forms of $M({\bf F}) \subset \BP V_{\bf F}$ at $o$ is isomorphic to the symbol system ${\bf F}$. \end{itemize} Conversely, any Euler-symmetric projective variety is of the form
$M({\bf F})$ for some symbol system ${\bf F}$ on a vector space $W$. \end{theorem}

\begin{proof}
 Viewing an element  $f^{k-j} \in (F^{k-j})^*$, $k>j \geq 0$,  as a linear map $F^{k-j} \to \mathbb{C},$ we define the composition $f^{k-j} \circ \iota_w^j$ as an element in $(F^k)^*$.
Using this, we  define an action of $W$ on $\BP V_{\bf F}$ as follows.
For $v \in W$ and $z= [t:w:f^2:\cdots : f^r] \in \BP V_{\bf F}$ with $f^k \in (F^k)^*$, define
$$
g_v \cdot z :=
[t: w + t v: g^z_v \cdot f^2: \cdots: g^z_v \cdot f^r]$$ where for each $2 \leq k \leq r$,  $$g^z_v \cdot f^k:=   \sum_{l=2}^k \binom{k}{l}  f^l \circ \iota_v^{k-l}+ k \iota_w \circ \iota_v^{k-1} + t \iota_v^{k}.$$
We have $g_0 \cdot z = z$ for all $z \in \BP V_{\bf F}$.

We claim $g_{v+u} \cdot z = g_{u} \cdot ( g_{v} \cdot z)$ for all $u, v \in W$.  By definition,
\begin{eqnarray*}
g_{u} \cdot( g_{v} \cdot z)   &= & g_{u} \cdot[t: w + t v: \cdots:  g^z_v \cdot f^k : \cdots]
\\ & = & [t: w + tv + tu: \cdots : g^{(g_v \cdot z)}_u \cdot (g^z_v \cdot f^k) : \cdots ] \end{eqnarray*} where
\begin{eqnarray*}
g^{(g_v \cdot z)}_{u} \cdot ( g^z_{v} \cdot f^k) &=&  \sum_{j=2}^k \binom{k}{j} \big( \sum_{l=2}^{j} \binom{j}{l} f^{l} \circ \iota_{v}^{j-l}+ j \iota_w \circ \iota_v^{j-1} + t \iota_v^j\big) \circ \iota_{u}^{k-j} \\
& & + k \iota_{w+tv} \circ \iota_u^{k-1} + t \iota_u^k. \end{eqnarray*}
By Lemma \ref{l.dot}, this can be simplified to
\begin{eqnarray*}
& &  \sum_{j=2}^k  \sum_{l=2}^{j} \binom{k}{j} \binom{j}{l} f^{l} \circ \iota_{v}^{j-l} \circ \iota_{u}^{k-j}  \\ & &  + \sum_{j=2}^k \binom{k}{j} (j \iota_w \circ \iota_v^{j-1}\circ \iota_{u}^{k-j} + t \iota_v^j \circ \iota_{u}^{k-j})  \  + k \iota_{w+tv} \circ \iota_u^{k-1} + t \iota_u^k \\
& =& \sum_{j=2}^k  \sum_{l=2}^{j} \binom{k}{j} \binom{j}{l} f^{l} \circ \iota_{v}^{j-l} \circ \iota_{u}^{k-j} \\ & &  + k \sum_{j=0}^{k-1} \binom{k-1}{j} \iota_w \circ \iota_v^j \circ \iota_u^{k-1-j} + t \sum_{i=0}^k \binom{k}{i} \iota_v^i \circ \iota_u^{k-i} \\
&=& \sum_{j=2}^k  \sum_{l=2}^{j} \binom{k}{j} \binom{j}{l} f^{l} \circ \iota_{v}^{j-l} \circ \iota_{u}^{k-j} + k \iota_w \circ \iota_{v+u}^{k-1} + t \iota_{v+u}^k \\
&=& \sum_{l=2}^k f^l \circ (\sum_{j=l}^k \binom{k}{j} \binom{j}{l}  \iota_{v}^{j-l} \circ \iota_{u}^{k-j})+ k \iota_w \circ \iota_{v+u}^{k-1} + t \iota_{v+u}^k \\
&=& \sum_{l=2}^k f^l \circ (\sum_{i=0}^{k-l} \binom{k}{l} \binom{k-l}{i} \iota_v^i \circ \iota_u^{k-l-i})+ k \iota_w \circ \iota_{v+u}^{k-1} + t \iota_{v+u}^k \\
&=& \sum_{l=2}^k \binom{k}{l}  f^l \circ \iota_{v+u}^{k-l}+ k \iota_w \circ \iota_{v+u}^{k-1} + t \iota_{v+u}^{k} \\
& = & g^z_{u+v} \cdot f^k
\end{eqnarray*}
This proves the claim, which verifies that $v \mapsto g_v$ is an action of $W$ on $\BP V_{\bf F}$.

Now we show that this action of $W$ preserves $M({\bf F})$.
Take $[t:w] \in \BP(\mathbb{C} \oplus W)$ general and $v \in W$, then $g_v \cdot \phi_{{\bf F}}([t:w])$ is equal to
\begin{eqnarray*}
& & g_v \cdot ([t^r: t^{r-1}w: t^{r-2}\ \iota_w^2:  \cdots: t^{r-k} \ \iota_w^k: \cdots: t\ \iota_w^{r-1}: \iota_w^r])\\
&=& [t^r: t^{r-1} (w+tv): \cdots: g^{\phi_{{\bf F}}([t:w])}_v \cdot  (t^{r-k} \iota_w^k): \cdots]
\end{eqnarray*}  where
$$g^{\phi_{{\bf F}}([t:w])}_v \cdot  (t^{r-k} \iota_w^k) =
  \sum_{l=2}^k \binom{k}{l}  t^{r-l} \iota_w^l \circ \iota_v^{k-l}+ k \iota_{t^{r-1}w} \circ \iota_v^{k-1} + t^r \iota_v^{k}.$$ By Lemma \ref{l.dot}, this is equal to $t^{r-k} \iota_{w+tv}^k$, yielding
  $$ g_v \cdot \phi_{{\bf F}}([t:w]) =  [t^r: t^{r-1} (w+tv): \cdots: t^{r-k} \ \iota_{w+tv}^k : \cdots] =  \phi_{{\bf F}}([t: w+tv]).$$
This implies that the action of $W$ on $\BP V_{\bf F}$ preserves $M({\bf F})$ and $\phi_{{\bf F}}$ is $W$-equivariant. Note that $\phi_{{\bf F}}$ maps
$W \subset \BP(\mathbb{C} \oplus W)$ biregularly to an open subset of $M({\bf F})$. This completes the proof of (i).

Consider the $\mathbb{C}^{\times}$-action on  $\BP(\mathbb{C} \oplus W)$  given by $\lambda \cdot [t:w] = [t: \lambda w]$.
This action induces a $\mathbb{C}^{\times}$-action on $\BP V_{\bf F}$ by
$$
\lambda \cdot [t:w:f^2:\cdots:f^r] = [t: \lambda w: \lambda^2 f^2: \cdots: \lambda^r f^r].
$$
It follows that $\lambda \cdot \phi_{{\bf F}}([t:w]) = \phi_{{\bf F}}(\lambda \cdot [t:w])$, hence $M({\bf F})$ is $\mathbb{C}^{\times}$-invariant.
This $\mathbb{C}^{\times}$-action has an isolated fixed point at $o$ and it acts on $T_o M({\bf F})$ as the scalar multiplication. Thus it is a $\mathbb{C}^{\times}$-action of Euler type at $o$.  Now we can use the $W$-action to translate this $\mathbb{C}^{\times}$-action to any point in this open orbit, hence $M({\bf F})$ is Euler-symmetric, proving (ii).

On an open subset, the variety $M({\bf F})$ is the graph of the map $w \mapsto (\iota^2_w, \cdots, \iota^r_w)$. By Lemma \ref{l.FF}, this shows that  the system of fundamental forms of $M({\bf F}) \subset \BP V_{\bf F}$ is isomorphic to the symbol system ${\bf F}$, proving (iii).

Finally,  for an Euler-symmetric projective variety $Z \subset \BP V$, let ${\bf F}$ be the symbol system isomorphic to the system of fundamental forms of $Z$ at a general point. Then $Z \subset \BP V$ is isomorphic to  $M({\bf F}) \subset \BP V_{\bf F}$ by Proposition  \ref{p.isom}.
\end{proof}

\begin{definition}\label{d.rank}
An Euler-symmetric projective variety $Z \subset \BP V$ has {\em rank} $r$ if its $(r+1)$-st fundamental form at a general point is zero. Equivalently, the Euler-symmetric projective variety $M({\bf F})$ associated to a symbol system ${\bf F}$ has rank $r$ if the symbol system ${\bf F}$ has rank $r$. \end{definition}

\begin{example}\label{e.rank2}  A symbol system of rank 2 is just a subspace of $\Sym^2 W^*$.  In this case,  our $M({\bf F})$ is reduced to the varieties constructed in Section 3 of \cite{La}. Euler-symmetric projective varieties of rank 2 are exactly quadratically symmetric varieties in \cite{FH2}. A complete classification of nonsingular Euler-symmetric varieties of rank 2 are given in  Theorem 7.8 of \cite{FH2}. \end{example}

\begin{example}\label{e.P}
For a nonzero homogeneous polynomial $P \in \Sym^r W^*$ of degree $r$, we can define the symbol system ${\bf F}_P$ of rank $r$ by setting $F^r = \langle P \rangle$  and  $$F^{r-j} = \langle \iota_{w_1} \circ \cdots \circ \iota_{w_j} P, \ w_1, \ldots, w_j \in W\rangle$$ for all $1 \leq j \leq r-2.$ When $r=3$ and $P$ is nondegenerate in a suitable sense, the variety $M({\bf F}_P)$ is exactly the projective Legendrian variety studied in Section 4.3 of \cite{LM}. \end{example}

\begin{example} \label{c.dim1}
An Euler-symmetric projective curve of rank $r$ is the rational normal curve in $\mathbb{P}^r$.
\end{example}

\begin{example}
 By Theorem \ref{t.model}, a nonsingular Euler-symmetric projective surface is biregular to a successive blow-ups of $\mathbb{P}^2$ or a Hirzebruch surface $\mathbb{F}_n$ along fixed points of the $\mathbb{G}_a^2$-action.
Their classification as projective surfaces does not seem straight-forward.
\end{example}

\begin{example}\label{e.HSS}
A rational homogeneous projective variety is Euler-symmetric if and only if it is Hermitian symmetric. In fact, if $G/P$ is Euler-symmetric, then  by Theorem \ref{t.model}, it is  an equivariant compactification of a vector group, hence it is Hermitian symmetric by \cite{A}. Conversely, all Hermitian symmetric spaces, under  projective embeddings equivariant with respect to their automorphism groups, are  Euler-symmetric.
\end{example}

\begin{example} \label{e.CI}
A nondegenerate hypersurface is Euler-symmetric if and only if it is a nondegenerate hyperquadric. In fact, if $M({\bf F}) \subset \BP(V_{\bf F})$ is a hypersurface, then ${\bf F}$ is of rank 2 and $F^2 = \langle Q \rangle$. If we write the coordinates of $\BP V_{\bf F}$ as $[z_0:z_1:\cdots: z_n: u]$, then $M({\bf F})$ is the hyperquadric $z_0 u = Q(z_1, \cdots, z_n)$. Conversely, if $Z \subset \BP V$ is a nondegenerate hyperqaudric, then up to a change of coordinates, $Z$ is given by the equation $z_0 u = Q(z_1, \cdots, z_n)$, then it is exactly the $M({\bf F})$ with ${\bf F}$ of rank 2 and $F^2 = \langle Q \rangle$.
\end{example}

\section{Base loci of a symbol system}

By Theorem \ref{t.model}, there is a natural 1-to-1 correspondence between  symbol systems and  Euler-symmetric projective varieties. It is interesting to investigate how the algebraic properties of a symbol system are reflected in the geometric properties of the corresponding Euler-symmetric projective variety, and vice versa. In this section, we look at this problem through the base loci of a symbol system and rational curves on the Euler-symmetric projective variety.

\begin{definition}
For a symbol system ${\bf F} = \oplus_{k\geq 0} F^k$ of rank $r$, define the projective algebraic subset
${\bf Bs}(F^k) \subset \BP W$ by the affine cone in $W$ $$\{
w \in W | \varphi(w, \cdots, w) = 0 \mbox{ for all } \varphi \in F^k\}.$$
By the definition of a symbol system, we have the inclusion $
 {\bf Bs}(F^k) \subset {\bf Bs}(F^{k+1})$ for each $k \in \N$.
The {\em order} of the symbol system ${\bf F}$ is the largest natural number $m$ such that
${\bf Bs}(F^m) = \emptyset$. As ${\bf Bs}(F^1) = {\bf Bs}(W^*) = \emptyset$ and
${\bf Bs}(F^{r+1}) = \BP W$, the order is less than of equal to the rank $r$ of the symbol system. The {\em base loci} of ${\bf F}$ is the nonempty projective algebraic subset
${\bf Bs}({\bf F}) :=  {\bf Bs}(F^{m+1})$  in $\BP W$, where $m$ is the order of ${\bf F}$.   \end{definition}


\begin{proposition} \label{p.Veronese}
If the order of ${\bf F}$ is equal to the rank of ${\bf F}$,  then the normalization of $M({\bf F})$ is projective space.
 If furthermore $M({\bf F})$ is nonsingular, then $M({\bf F})$ is a biregular projection of the $r$-th Veronese variety. \end{proposition}

 \begin{proof}
 Note that the birational map $\phi_{\bf F} : \BP (\C \oplus W) \dasharrow M({\bf F})$ in Definition \ref{d.model} has base locus ${\bf Bs}(F^r)$. Our assumption says that ${\bf Bs}(F^r) = \emptyset$. Thus $\phi_{\bf F}$ is a birational  morphism, which is finite over its image because it contracts no curves on $\BP (\C \oplus W)$. This implies that $\phi_{\bf F}$ is the normalization map. If $M({\bf F})$ is nonsingular, then $\phi_{\bf F}$ is an isomorphism. Note that $M({\bf F}) \subset \BP(V_{\bf F})$
 is the linear projection from the $r$-th Veronese embedding $\BP(\mathbb{C} \oplus W) \to \BP(\Sym^r (\mathbb{C} \oplus W))$, which is biregular as $\phi_{\bf F}$ is an isomorphism.
\end{proof}

\begin{proposition} \label{p.orbit}
Let $A_o \subset \GL(V_{\bf F})$ be a multiplicative subgroup corresponding to a $\mathbb{C}^{\times}$-action on $M({\bf F})$ of Euler type at $o:=[1:0:\cdots:0] \in M({\bf F})$, from Theorem \ref{t.model} (ii). In terms of the natural identification of  $T_o M({\bf F})$ and $W$ from Definition \ref{d.model} ( uniquely determined up to a scalar multiple), we have  the following.  \begin{itemize} \item[(1)] If the $A_o$-stable curve on $M({\bf F})$ through $o$  in the direction of a nonzero vector $w \in W$ has degree $\leq k$, then $w \in {\bf Bs}(F^{k+1})$. \item[(2)] Let $m$ be the order of ${\bf F}$.  Then
for any nonzero vector $w$ belonging to ${\bf Bs}(F^{m+1})$, the $A_o$-stable curve on $M({\bf F})$ through $o$ in the direction of $w$ is a rational normal curve of degree $m$. \end{itemize} It follows that the rank of ${\bf F}$ is the maximal degree of $A_o$-stable curves through $o$ and the order of ${\bf F}$ is the minimal degree of $A_o$-stable curves through $o$. \end{proposition}

\begin{proof} From Definition \ref{d.model} and  Theorem \ref{t.model} (ii), the $A_o$-stable curve in the direction of a nonzero vector $w \in W$
is  the closure of the curve
$$
[1: \lambda w: \lambda^2 \iota_w^2: \cdots: \lambda^r \iota_w^r],  \ \lambda \in \mathbb{C}^{\times}.
$$
Both (1) and (2) follow immediately from the above expression.
\end{proof}

\begin{proposition} \label{p.vmrt}
If $M({\bf F})$ is nonsingular, then the base loci ${\bf Bs}({\bf F}) $ is nonsingular. \end{proposition}

\begin{proof}
Let $A_o \subset \GL(V_{\bf F})$ be as in Proposition \ref{p.orbit}.
 Let $\sR^k$ be the set of all rational curves of degree $k$ through $o$ on
 $M({\bf F})$. Assume that $\sR^{k-1} = \emptyset$ and $\sR^k \neq \emptyset$.
 Then $\sR^k$ is a complete variety and $A_o$ acts on each irreducible component of $\sR^k$ with nonempty fixed points.
 From Proposition \ref{p.orbit}, the order of ${\bf F}$ is $k$ and a general member of each irreducible component of $\sR^k$ is a rational normal curve of degree $k$.

 Since members of $\sR^k$ are rational curves of minimal degree through $o$ on the nonsingular projective variety $M({\bf F})$, a general member $C$ of each irreducible component  of  $\sR^k$ has normal bundle of type  $\sO(1)^p \oplus \sO^q$ for some nonnegative integers $p$ and $q$ (e.g. Proposition 6 in \cite{HM98}).
 We claim that $C$ is $A_o$-stable.  Otherwise, we have a nontrivial family of members $C_t, t \in A_o$, of
$\sR^k$. From the type of the normal bundle, the tangent directions of curves $C_t$ at $o$ are distinct. This is impossible because $A_o$ acts trivially on $\BP T_o M({\bf F})$.
This proves the claim.

Since $C$ is a general member of $\sR^k$, the claim implies that
 all members of $\sR^k$ are $A_o$-stable.
 From Bialynicki-Birula's structure theory (in \cite{BB}) of $\C^{\times}$-actions on nonsingular projective varieties and
the fact that members of $\sR^k$ are exactly $A_o$-stable curves of minimal degree through $o$, we see that the set of the tangent directions to members of $\sR^k$ is nonsingular, being
biregular to some components of the fixed point set of $A_o$-action on $M({\bf F})$.
By Proposition \ref{p.orbit}, this set is exactly ${\bf Bs}({\bf F})$.
 \end{proof}

\begin{definition}\label{d.saturated}
A symbol system ${\bf F}$ of rank $r$ and of order 1 is {\em saturated}
if  $$H^0(\mathbb{P}W, \mathcal{I}_{{\bf Bs}({\bf F})} \otimes \sO(2)) = F^{2} \mbox{ and } F^{k+1} = {\bf prolong}(F^k)$$ for all $k \geq 2.$ \end{definition}

\begin{proposition}\label{p.saturated}
Let ${\bf F}$ be a saturated symbol system. If the Picard group of $M({\bf F})$ is discrete (e.g. if $M({\bf F})$ is normal), then the embedding $M({\bf F}) \subset \BP V_{\bf F}$ is linearly normal. \end{proposition}

\begin{proof}
Let $L$ be the hyperplane line bundle of $\BP V_{\bf F}$ restricted to $M({\bf F})$.
Let $\widetilde{V}$ be the dual space of $H^0(M({\bf F}), L)$ and $j: M({\bf F}) \subset \BP \widetilde{V}$ be the linearly normal embedding.
Since the Picard group is discrete, the connected automorphism group ${\rm Aut}_o(M({\bf F}))$ acts on $\BP \widetilde{V}$ and a $\C^{\times}$-action of Euler type  at a general point $x \in M({\bf F}) \subset \BP V_{\bf F}$ induces a $\C^{\times}$-action of Euler type  for the embedding $j$.
It follows that the image $j(M({\bf F})) \subset \BP \widetilde{V}$ is an Euler-symmetric variety.
Let $\widetilde{\bf F} = (\widetilde{F}^k \subset \Sym^k W^*)$ be the associated symbol system.
By Proposition \ref{p.orbit}, we have the equality ${\bf Bs}(\widetilde{\bf F}) = {\bf Bs}({\bf F})$. The saturatedness gives the inclusion
$$\widetilde{F}^2 \subset H^0(\mathbb{P}W, \mathcal{I}_{{\bf Bs}({\bf F})} \otimes \sO(2)) = F^2$$
and successive inclusions
$$\widetilde{F}^{k+1} \subset {\bf prolong}(\widetilde{F}^k) \subset {\bf prolong}(F^k) = F^{k+1}.$$ It follows that $\dim \widetilde{V} \leq \dim V_{\bf F}$, which implies $V_{\bf F} = \widetilde{V}$.
This shows that $M({\bf F}) \subset \BP V_{\bf F}$ is linearly normal.

It remains to show that if $M({\bf F})$ is normal, then ${\rm Pic}(M({\bf F}))$ is discrete. Let $S = {\rm Sing}(M({\bf F}))$, which is of codimension $\geq 2$. Let $U = M({\bf F}) \setminus S$ be the smooth locus, which contains an open subset $U_0$ isomorphic to $W$. For $n= \dim W$, the Chow group $CH_{n-1}(U \setminus U_0)$ is discrete because it is  generated by the irreducible components of $U \setminus U_0$ that have dimension $n-1$.  By the localization exact sequence of Chow groups
$$CH_{n-1}(U \setminus U_0) \to  CH_{n-1}(U) \simeq {\rm Pic}(U) \to CH_{n-1}(U_0) \simeq {\rm Pic}(U_0) = 0,$$ we see that ${\rm Pic}(U)$ is discrete.   As $M({\bf F})$ is normal, the Picard group ${\rm Pic}(M({\bf F}))$ is identified with Cartier divisor classes on $M({\bf F})$, hence it is a subgroup of the (Weil) divisor class group ${\rm Cl}(M({\bf F}))$. As  $S$ has codimension $\geq 2$, we have ${\rm Cl}(M({\bf F})) = {\rm Cl}(U) = {\rm Pic}(U)$ which is discrete, proving the claim.
\end{proof}

\begin{example}\label{e.IHSS}
The minimal projective embedding of an irreducible Hermitian symmetric space (also called minuscule varieties), other than projective space,  is given by a saturated symbol system (see Theorem 3.1 and Corollary 3.6 of \cite{LM03}).
\end{example}

\begin{remark}\label{r.prime}
Is the converse of Proposition \ref{p.saturated} true under some geometric conditions on
$M({\bf F})$?
Recall that a {\em prime Fano manifold} is a nonsingular projective subvariety $X \subset \BP^N$ covered by lines such that ${\rm Pic}(X) \simeq \Z \langle \mathcal{O}_X(1) \rangle$.
The projective subvarieties in Example \ref{e.IHSS} are examples of prime Fano manifolds.
A natural question is: if $M({\bf F}) \subset \BP V_{\bf F}$ is a linearly normal prime Fano manifold,
is ${\bf F}$ saturated? The following example in \cite{PR} is not a prime Fano manifold. \end{remark}

\begin{example} \label{e.PR}
Let $W$ be a vector space of dimension $n$ with coordinates $x_1, \cdots, x_n$.
Consider the symbol system ${\bf F}$ defined by $$F^2 = \langle x_1^2, x_1x_2, \cdots, x_1x_n \rangle \mbox{ and } F^3=\langle x_1^3 \rangle.$$ The base loci
is the hyperplane $\{ x_1=0\}$ in $\mathbb{P} W$, hence $H^0(\mathbb{P}W, \mathcal{I}_{{\bf Bs}({\bf F})} \otimes \sO(2)) = F^{2}$. But ${\bf F}$ is not saturated because $$ x_1^3, x_1^2x_2, \cdots, x_1^2x_n \ \in \ {\bf prolong}(F^2).$$   The associated Euler-symmetric variety $M({\bf F})$ is a rational normal scroll (\cite{PR}, Theorem 5.2), which is linearly normal.
\end{example}

\section{Equivariant compactifications of vector groups}

By Theorem \ref{t.model}, every Euler-symmetric projective variety $Z$ is an equivariant compactification of a vector group, i.e., there exists an action of the vector group $W, \dim W= \dim Z,$ on $Z$ with an open orbit.   Is the converse also true? In \cite{HT} (Section 4.1), it is  pointed out that there exists an equivariant compactification of $\mathbb{G}_a^1$, a singular curve,  which does not admit $\mathbb{G}_a^1$-equivariant projective embeddings. By \cite{B}, there are many singular cubic  hypersurfaces which are equivariant compactifications of vector groups, but by Example  \ref{e.CI}, they are not Euler-symmetric. Thus it seems reasonable to exclude singular varieties.
We propose the following conjecture.

\begin{conjecture}\label{c.Fano}
Let $X$ be a Fano manifold of Picard number 1 which is an equivariant compactification of a vector group. Then $X$ can be realized as an Euler-symmetric projective variety under a suitable projective embedding. \end{conjecture}

In this section, we show that the conjecture holds under one technical assumption, formulated in terms of VMRT (varieties of minimal rational tangents).

\begin{definition}\label{d.VMRT}
Let $X$ be a uniruled projective manifold. An irreducible
component $\sK$ of the space of rational curves
on $X$ is called {\em a minimal rational component} if the
subscheme $\sK_x$ of $\sK$ parameterizing curves passing through
a general point $x \in X$ is non-empty and proper. Curves
parameterized by $\sK$ will be called {\em minimal rational
curves}. Let $\rho: \sU \to \sK$ be the universal family and $\mu:
\sU \to X$ the evaluation map. The tangent map $\tau: \sU
\dasharrow \BP T(X)$ is defined by $\tau(u) = [T_{\mu(u)}
(\mu(\rho^{-1} \rho(u)))] \in \BP T_{\mu(u)}(X)$. The closure $\sC
\subset \BP T(X)$ of its image is the {\em VMRT-structure} on $X$. The
natural projection $\sC \to X$ is a proper surjective morphism and
a general fiber $\sC_x \subset \BP T_x(X)$ is called the VMRT  at the
point $x \in X$. The VMRT-structure $\sC$ is  {\em locally
flat} if there exists  an analytical open subset $U$ of $X$ with an open immersion
  $\phi: U \to
 \C^n, n= \dim X,$ and a projective subvariety $Y \subset \BP^{n-1}$
with $\dim Y = \dim \sC_x$ such that $\phi_*: \BP T(U) \to \BP
T(\C^n)$ maps $\sC|_{U}$ into the trivial fiber subbundle $\C^n
\times Y$ of the trivial projective bundle $\BP T(\C^n) = \C^n \times \BP^{n-1}.$
\end{definition}

 The concept of VMRT is useful to us via the Cartan-Fubini type extension theorem, Theorem 1.2 in \cite{HM01}.
We will quote a simpler version, Theorem 6.8 of \cite{FH1}.

 \begin{theorem}\label{t.CF}
 Let $X_1, X_2$ be two Fano manifolds of Picard number 1, different from projective spaces. Let $\sK_1$ (resp. $\sK_2$) be a family of minimal
 rational curves on $X_1$ (resp. $X_2$). Assume that for a general point $x \in X_1$, the VMRT $\sC_x \subset
 \BP T_x(X_1)$ is irreducible and nonsingular. Let $U_1 \subset X_1$ and $U_2 \subset X_2$ be connected analytical open subsets.
 Suppose there exists a biholomorphic map $\varphi: U_1 \to U_2$ such that for a general point $x \in U_1$, the
 differential $d\varphi_x: \BP T_x(U_1) \to \BP T_{\varphi(x)}(U_2)$ sends $\sC_x$ isomorphically to $\sC_{\varphi(x)}$. Then there exists a biregular morphism $\Phi: X_1 \to X_2$ such that $\varphi= \Phi|_{U_1}.$ \end{theorem}

 Now we study some basic properties of equivariant compactifications of vector groups of Picard number one.

\begin{proposition}\label{p.vmrtSEC}
Let $X$ be a Fano manifold of Picard number one which is an equivariant compactification of the vector group $W$ with an open orbit $X^o \subset X$.   Let $\sK$ and $\sC$ be as in Definition \ref{d.VMRT}. Assuming
that  the VMRT $\sC_x$ at  a point $x \in X^o$ is nonsingular, we have the following.
\begin{itemize}
 \item[(i)] The VMRT-structure $\sC$ is locally flat.
\item[(ii)] The VMRT $\sC_x \subset \BP T_x X$ at a general point $x \in X$ is
nondegenerate and irreducible.
\item[(iii)]  The $\C^{\times}$-action on the vector space $W$ by scalar multiplication induces a $\C^{\times}$-action on $X$.
\item[(iv)] A member of $\sK_x, x \in X^o,$ is the closure of the image of a 1-dimensional subspace in $W$.
\item[(v)] Let $D \subset X$ be the complement of the open orbit $X^o$.
Then $D$ is an irreducible divisor which is an ample generator of ${\rm
Pic}(X)$. If $C \subset X$ is a minimal rational curve, then $D
\cdot C = 1$.
\item[(vi)] The map $\phi_{|D|}: X \dasharrow \BP H^0(X, D)^*$ is
birational onto its image, which sends a general member of $\sK$
on $X$ to a line in $\BP H^0(X, D)^*$. \end{itemize}
\end{proposition}

\begin{proof}
(i) The $W$-action  on $\BP T X$ preserves the VMRT-structure $\sC \subset \BP T X$. Thus the $W$-action on $X^o$ trivializes $\sC|_{X^o}$ as a subbundle of $\BP T X^o \cong \BP T W.$

(ii) By Proposition 2.2 \cite{FH0}, the VMRT $\sC_x$ is irreducible. If $\sC_x
\subset \BP T_x X$ is degenerate, then the distribution spanned by
the VMRT's on $X^o$ is integrable since $\sC$ is locally flat. This
contradicts Proposition 13 in \cite{HM98}.

(iii) The induced $\C^{\times}$-action on $X^o$ preserves $\sC|_{X^o} \subset \BP T X^o$
because $\sC|_{X^o}$ corresponds to a trivial subbundle of $\BP T W$ by (i). By
the assumption that $\sC_x$ is nonsingular and (ii), we can apply Theorem \ref{t.CF} to conclude that this $\C^{\times}$-action extends to a $\C^{\times}$-action on $X$.

(iv) As we have seen in the proof of Proposition \ref{p.vmrt}, a  member of
$\sK_x, x \in X^o$, is stable under the  $\C^\times$-action of (iii).  Thus it is the closure of the image of
a 1-dimensional subspace in $W$.

(v) As $X$ has Picard number one, $D$ is irreducible and it freely
generates ${\rm Pic}(X)$ by Theorem 2.5 \cite{HT}. Take a
codimension 1 linear subspace $\C^{n-1} \subset W$ and let $H$ be
its closure in $X$. The intersection $H \cap D$ has codimension 2
in $X$. Thus a general minimal rational curve $C$ on $X$
intersects $H$ in $H \cap X^o$ (e.g. by Lemma 1.1 (3) in \cite{HM01}). By (iv), we conclude $H \cdot C =1$. This
implies $H$ generates ${\rm Pic}(X)$ and  $D \cdot C =1$.

(vi) From the proof of (v), the closure of the image in $X^o$ of any vector subspace $\C^{n-1} \subset W$
is linearly equivalent to $D$. Thus $\phi_{|D|}$ gives an
embedding from $X^o$ to $\BP H^0(X, D)^*$.
\end{proof}

\begin{theorem} \label{t.EulerSEC}
Let $X$ be a Fano manifold of Picard number 1 which is an equivariant compactification of a vector group and whose VMRT at a general point is nonsingular. Let $D \subset X$ be as in Proposition \ref{p.vmrtSEC} (v) and let $m$ be the minimal number such that $mD$ is very ample. Then the embedding $X \subset \BP H^0(X, mD)$ realizes $X$ as an Euler-symmetric projective variety whose system of fundamental forms has order $m$.
\end{theorem}
\begin{proof}
By the proof of Proposition \ref{p.vmrtSEC} (iii), for a general point $x \in X$, there exists a $\mathbb{C}^{\times}$-action on  $X \subset \BP H^0(X, mD)^*$  of Euler type at $x$. Thus the embedded variety is Euler-symmetric.   By Proposition \ref{p.vmrtSEC} (v), the minimal degree of $\C^{\times}$-orbits  is $m$. By Proposition \ref{p.orbit}, the order of the system of fundamental forms has order $m$.
\end{proof}

Recall (Remark \ref{r.prime}) that a prime Fano manifold $X$ is a nonsingular projective variety $X$ with
${\rm Pic}(X) \cong \Z L$ for a very ample line bundle $L$ such that $X$ is covered by rational curves of degree $1$ with respect to $L$. It is well-known that in this case the VMRT of lines through a general point is nonsingular (e.g. Proposition 3.2 in \cite{FH1}). Thus we have the following corollary.

\begin{corollary} \label{c.SEC}
If a prime Fano manifold $X$ is an equivariant compactification of the vector group, then the embedding $X \subset \BP H^0(X, L)^*$ is  Euler-symmetric of order 1. \end{corollary}

Based on Theorem \ref{t.EulerSEC}, we propose the following, which would imply Conjecture \ref{c.Fano}.

\begin{conjecture}
Let $X$ be a Fano manifold of Picard number 1 which is an equivariant compactification of a vector group. Then for some choice of $\sK$, the VMRT $\sC_x$ at a general point $x \in X$ is nonsingular. \end{conjecture}

\bigskip
Baohua Fu

Institute of Mathematics, AMSS, Chinese Academy of Sciences,

55 ZhongGuanCun East Road, Beijing, 100190, China
and

 School of Mathematical Sciences, University of Chinese Academy of Sciences, Beijing, China

 bhfu@math.ac.cn

\bigskip
Jun-Muk Hwang

 Korea Institute for Advanced Study, Hoegiro 85,

Seoul, 02455, Korea

jmhwang@kias.re.kr

\end{document}